\documentclass[10pt]{article}

\usepackage{calc}
\usepackage{color}
\usepackage{amsmath}
\usepackage{amssymb}
\usepackage{amsthm}
\usepackage[sans]{dsfont}
\usepackage{mathtools}
\usepackage{alphalph}

\theoremstyle{plain}
\newtheorem{theorem}{Theorem}

\newtheorem{lemma}{Lemma}

\theoremstyle{definition}

\newtheorem{remark}[lemma]{Remark}

\usepackage{perpage}
\MakePerPage{equation}
\renewcommand{\theequation}{\theperpage\alphalph{\value{equation}}}
\newcommand{\myref}[1]{\ref{#1}\relax}
\newcommand{\myeqref}[1]{(\ref{#1})\relax}
\newcommand{\mylabel}[1]{\label{#1}\relax}
\newcommand{\myeqlab}[1]{\tag{\theequation}\stepcounter{equation}\label{#1}}

\newcommand{\cartp}{\times}
\newcommand{\R}{\mathds{R}}
\newcommand{\T}{\mathds{T}}
\newcommand{\ndiv}{\nabla\dotp}
\newcommand{\dotp}{\cdot}
\newcommand{\vv}{\vec v}
\newcommand{\pt}{\partial_t}
\newcommand{\subeq}[2]{\mathord{\underbrace{\mathop{#1}}_{#2}}}
\newcommand{\topref}[2]{\overset{\text{\myeqref{#1}}}{#2}}
\newcommand{\toprefb}[3]{\overset{\text{\myeqref{#1}}}{\underset{\text{\myeqref{#2}}}{#3}}}
\newcommand{\defm}[1]{\emph{#1}}

\newcommand{\Uz}{\overline U}
\newcommand{\eps}{\epsilon}
\newcommand{\dnconv}{\searrow}
\newcommand{\conv}{\rightarrow}
\newcommand{\BV}{\operatorname{BV}}
\newcommand{\selset}[2]{\{#1~:~#2\}}
\newcommand{\pd}[1]{\partial_{#1}}
\newcommand{\supp}{\operatorname{supp}}
\newcommand{\Dfun}{\mathcal{D}} 
\newcommand{\Ddist}{\Dfun'} 
\newcommand{\Leb}{\mathcal L}
\newcommand{\Leba}[1]{\Leb^{#1}}
\newcommand{\Linf}{\Leba\infty}
\newcommand{\embed}{\hookrightarrow}
\newcommand{\Lap}{\Delta}
\newcommand{\curl}{\operatorname{curl}}
\newcommand{\csep}{\quad,\quad}
\newcommand{\impl}{\Rightarrow}
\newcommand{\const}{\text{const}}
\newcommand{\visc}{\eps}

\renewcommand{\vv}{\mathbf v}
\newcommand{\vx}{x}

\newcommand{\ent}{\eta}
\newcommand{\fen}{\psi}
\newcommand{\dens}{\varrho}
\newcommand{\piv}{\pi}
\newcommand{\Dt}{D_t}
\newcommand{\cset}{F}
\newcommand{\coneinf}{K}

\begin{document}

\title{Relative entropy and compressible potential flow}
\author{Volker Elling\footnote{This material is based upon work partially supported by the
National Science Foundation under Grant No.\ NSF DMS-1054115 and by a Sloan Foundation Research Fellowship}}
\date{} 
\maketitle

%\begin{center}
%  {\it Dedicated to Tai-Ping Liu on the occasion of his seventieth birthday,\\with deep respect for his work and gratitude 
%    for his advice and support.}
%\end{center}

\begin{abstract}
  Ccompressible (full) potential flow is expressed as an equivalent first-order system of 
  conservation laws for density $\dens$ and velocity $\vv$. 
  Energy $E$ is shown to be the only nontrivial entropy for that system
  in multiple space dimensions, and it is strictly convex in $\rho,\vv$ if and only if $|\vv|<c$. 
  For motivation some simple variations on the relative entropy theme of Dafermos/DiPerna are given,
  for example that smooth regions of weak entropy solutions shrink at finite speed,
  and that smooth solutions force solutions of singular entropy-compatible perturbations to converge to them. 
  We conjecture that entropy weak solutions of compressible potential flow are unique, 
  in contrast to the known counterexamples for the Euler equations. 
\end{abstract}

\section{Full Euler}

\setcounter{equation}{1}

Consider the Euler equations for compressible inviscid fluids in $n$ space dimensions:
\begin{alignat*}{5}  0 &= \dens_t &&+ \nabla\cdot(\dens\vv) && \quad,\quad && \text{[mass]} \myeqlab{eq:mass} \\
    0 &= (\dens\vv)_t &&+ \nabla\cdot(\dens\vv\otimes\vv) &&+ \nabla p\quad,\quad && \text{[momentum]}  \myeqlab{eq:mom} \\
    0 &= E_t &&+ \nabla\cdot(E\vv) &&+ \ndiv(p\vv)\quad. \qquad &&\text{[energy]} \myeqlab{eq:energy} \end{alignat*}
$\dens$ is density, $\vv$ velocity, $E=\dens e$ energy density, 
\begin{alignat*}{5} e=q+\frac12|\vv|^2 \myeqlab{eq:edef} \end{alignat*} 
energy per mass, 
$q$ heat per mass. 
With $U=(\dens,\dens\vv,E)$ we obtain the short form
\begin{alignat*}{5}  0 &= U_t + \nabla\cdot\vec f(U)   \myeqlab{eq:claw}  \end{alignat*}
Consider smooth $U$ that are vacuum-free ($\dens>0$). 
For any field $f$ the transport equation
\begin{alignat*}{5} g &= \Dt f \quad\text{with $\Dt=\pt+\vv\dotp\nabla$} \end{alignat*} 
is equivalent to the conservation law
\begin{alignat*}{5} \dens g &= (\dens f)_t + \ndiv(f\dens\vv) \end{alignat*} 
so \myeqref{eq:mom} is equivalent to
\begin{alignat*}{5} 0 &= \Dt\vv + \dens^{-1}\nabla p \myeqlab{eq:vv} \end{alignat*} 
and \myeqref{eq:energy} is equivalent to
\begin{alignat*}{5} 0 &= \Dt e + \dens^{-1}\ndiv(p\vv)
=
\Dt q + \vv\dotp(\subeq{\Dt\vv + \dens^{-1}\nabla p}{\topref{eq:vv}{=}0}) + \dens^{-1}p\ndiv\vv
\myeqlab{eq:q} \end{alignat*}

\section{Entropy and isentropic Euler}

We seek $s=s(\dens,q)$ that satisfy the pure transport equation
\begin{alignat*}{5} 0 &= \Dt s
= s_\dens\Dt\dens + s_q\Dt q
\toprefb{eq:mass}{eq:q}{=} -(s_\dens\dens+s_q\dens^{-1}p)\ndiv\vv \end{alignat*} 
which is satisfied if $s=s(\dens,q)$ is a solution of 
\begin{alignat*}{5} 0 &= \dens s_\dens + \dens^{-1}ps_q \end{alignat*} 
This is solvable by the method of characteristics, although here and elsewhere some assumptions on $p$ are needed; the most common \defm{polytropic} pressure law
\begin{alignat*}{5}
    p(\dens,q) &= (\gamma-1)\dens q  \myeqlab{eq:p-full}
\end{alignat*}
with \defm{isentropic coefficient} $\gamma\in(1,\infty)$ yields a solution
\begin{alignat*}{5}
    s &= \log\frac{p}{\dens^\gamma} = \log q +(1-\gamma)\log\dens   \myeqlab{eq:s}
\end{alignat*}
While generalizations to other $p$ are certainly possible, we focus on the polytropic case.

$\Dt s=0$ means $s$ is constant on each particle path (integral curve of $(1,\vv)$ in space-time $(t,\vx)$).
If we assume $s$ is constant \emph{everywhere} at some time $t$, for example $t=0$, then it is constant for all $t$.
In this case we may solve for $p$ as a function of $s,\dens$.
The energy equation is redundant now, so we obtain the
\defm{isentropic Euler equations} (sometimes referred to as \defm{adiabatic}):
\begin{alignat*}{5}
    0 &= \dens_t + \nabla\cdot(\dens\vv) \myeqlab{eq:mass2}\\
    0 &= \vv_t + \vv\cdot\nabla\vv + \nabla\piv \myeqlab{eq:mom2}
\end{alignat*}
where $\piv=\piv(\dens)$ solves 
\begin{alignat*}{5} \piv_\dens(\dens) = \dens^{-1} p_\dens(\dens) \quad. \end{alignat*}
Linearization around a constant background $U=(\dens,0)$ yields the wave equation $\dens_{tt}-c^2\Delta\dens=0$ with
\defm{speed of sound}
\begin{alignat*}{5} c = \sqrt{p_\dens} \quad. \myeqlab{eq:c} \end{alignat*} 

\section{Vorticity and compressible potential flow}

Write \myeqref{eq:mom2} in coordinates with Einstein convention:
\begin{alignat*}{5}
0 &= v^k_t + v^jv^k_j + \pi_k  \myeqlab{eq:momtens}
\end{alignat*}
Assume smoothness and take the curl: for all $1\leq k<m\leq n$, with 
\begin{alignat*}{5}&
    \omega^{km}:=v^k_m-v^m_k , \myeqlab{eq:omega}
\end{alignat*}
we obtain
\begin{alignat*}{5}
    0 &= v^k_{tm} + v^jv^k_{jm} &&+ v^j_mv^k_j 
    \\&- (v^m_{tk} + v^jv^m_{jk} &&+ v^j_kv^m_j)
    \\&= \omega^{km}_t + v^j\omega^{km}_j &&+ v^j_m(v^j_k+\omega^{kj}) \\&&&- v^j_k(v^j_m+\omega^{mj})
    \\&= D_t\omega^{km} &&+ v^j_m \omega^{kj} - v^j_k \omega^{mj} \quad.
    \myeqlab{eq:vort-n}
\end{alignat*}
If $\omega=0$ at some time $t$, then \myeqref{eq:vort-n} guarantees 
\begin{alignat*}{5}&
    \omega = 0 \qquad\text{for all $t$.} \myeqlab{eq:vort0}
\end{alignat*}
Then
\begin{alignat*}{5}&
    \vv = \nabla\phi \myeqlab{eq:phi}
\end{alignat*}
where $\phi$ is the scalar \defm{velocity potential}. Moreover \myeqref{eq:momtens} becomes
\begin{alignat*}{5}
    0 &\topref{eq:vort0}= v^k_t + v^j(v^j_k + \subeq{\omega^{kj}}{=0}) + \pi_k 
    = v^k_t+(\frac12v^jv^j+\pi)_k = v^k_t + B_k
    \myeqlab{eq:vB}
\end{alignat*}
where 
\begin{alignat*}{5}&
B := \frac12|\vv|^2+\pi \quad.    \myeqlab{eq:Bern}
\end{alignat*}
Using \myeqref{eq:phi} we obtain 
\begin{alignat*}{5}& 0 = (\phi_t+\frac12|\nabla\phi|^2+\pi)_k, \end{alignat*}
hence
\begin{alignat*}{5}& C(t) = \phi_t+\frac12|\nabla\phi|^2+\pi, \end{alignat*}
which can be normalized to the \defm{Bernoulli relation}
\begin{alignat*}{5} \phi_t+\frac12|\nabla\phi|^2+\pi(\dens) &= 0  \myeqlab{eq:bernoulli} \end{alignat*}
by a transformation $\phi\leftarrow\phi+\int^t_0C(t)dt$ (which has no effect on $\vv=\nabla\phi$). We may solve for $\dens$:
\begin{alignat*}{5}&
    \dens = \dens(\phi_t,|\nabla\phi|) = \pi^{-1}\big( - \phi_t - \frac12|\nabla\phi|^2\big)  \myeqlab{eq:pirho}
\end{alignat*}
Now the momentum equations have been eliminated as well; all that remains is the continuity equation
\begin{alignat*}{5}&
\dens(\phi_t,|\nabla\phi|)_t + \nabla\cdot\big( \dens(\phi_t,|\nabla\phi|)\nabla\phi \big) = 0 \myeqlab{eq:potf}
\end{alignat*}
a scalar second-order quasilinear divergence-form PDE called \defm{compressible potential flow}
(also: ``\defm{full} potential flow'', in distinction from approximations like the transonic small disturbance equation).

\section{Admissibility}

Even for smooth initial data, solutions of the Euler equations and other hyperbolic systems of the form
\begin{alignat*}{5} 0 &= U_t + A^i(U) U_i \myeqlab{eq:Aform}\end{alignat*} 
usually form discontinuities in finite time 
\cite{lax-singularity-formation,john-singularity-formation,tpliu-singularity-formation,sideris}.
In some circumstances, particularly when $A^i=f^i_U$ arise from the conservation form
\begin{alignat*}{5} 0 &= U_t + f^i(U)_i \quad, \myeqlab{eq:claww}\end{alignat*} 
we can continue past the time of singularity formation by considering \defm{weak solutions}. 
For initial data $U_0$ they are defined as satisfying
\begin{alignat*}{5} 0 &= \int \chi(0,x)U_0(x)dx+ \iint \chi_t(t,x)U(t,x)+\chi_i(t,x)f^i(U(t,x))dx~dt \myeqlab{eq:claww-weakform} \end{alignat*} 
for any smooth \defm{test function} $\chi$ defined on $[0,\infty)\cartp\R^n$ with compact support. 
(Analogously, we define weak solutions of potential flow \myeqref{eq:potf} as Lipschitz-continuous $\phi$ that satisfy 
\begin{alignat*}{5} 0 &= \int \chi(0,x)\dens_0(x)dx + \iint \chi_t(t,x)\dens(t,x)+\chi_i(t,x)\dens(t,x)v^i(t,x)dx~dt
\myeqlab{eq:potf-weakform} \end{alignat*} 
for any smooth test function $\chi$, with $\dens$ as in \myeqref{eq:pirho}, as well as 
\begin{alignat*}{5} \phi(0,x) = \phi_0(x) \end{alignat*} 
where $\dens_0,\phi_0$ is prescribed initial data.)

The most basic discontinuities are located on a smooth hypersurface $S$, with $U$ smooth on each side. 
In that case \myeqref{eq:claww-weakform} is equivalent to \myeqref{eq:claww} on each side of $S$ combined with the 
\defm{Rankine-Hugoniot condition}
\begin{alignat*}{5}&
    [(U,f^1(U),...,f^n(U))\dotp(-\sigma,\nu^1,...,\nu^n)] = 0
\end{alignat*}
in each point $\vx\in S$, where $\nu$ is a unit normal to $S$ in $\vx$, $\sigma$ is the speed of $S$ in $\vx$,
and $[g]:=g_+-g_-$ where $g_\pm$ are the $\pm$-side limits, with $\nu$ pointing to the $+$ side.

It is well-known that the weak formulation does not determine solutions uniquely
and that some solutions may contain unphysical features such as \defm{expansion shocks}; 
for Burgers equation 
\begin{alignat*}{5} 0 &= u_t + (\frac12 u^2)_x \end{alignat*} 
constant initial data $u=0$ allows not only the strong solution $u=0$ but also the weak solution
\begin{alignat*}{5} u(t,x) = \begin{cases} 
  0 , & x/t < -\sigma \\ 
  -2\sigma , & -\sigma < x/t < 0 \\ 
  2\sigma , & 0 < x/t < \sigma \\ 
  0 , & \sigma < x/t 
\end{cases} \quad.\end{alignat*} 
Analogous solutions for full Euler are easily constructed by solving 1d Riemann problems. 
To achieve uniqueness it is necessary to impose additional \defm{admissibility criteria}.
\newcommand{\Ui}{U^\visc}
Most derive from the \defm{vanishing (uniform) viscosity criterion}: by analogy with passing to the Euler equations from Navier-Stokes,
we seek solutions $U$ of \myeqref{eq:claww} that are limits of solutions $\Ui$ of 
\begin{alignat*}{5} 0 &= \Ui_t + f^i(\Ui)_i - \visc \Ui_{ii} \myeqlab{eq:vanunifvisc}\end{alignat*} 
as $\visc\dnconv 0$. In case of bounded almost everywhere convergence, for example, the limit $U$ satisfies \myeqref{eq:claww} 
in the weak sense. 

Verifying the vanishing viscosity criterion directly is cumbersome. It is preferable to derive and utilize a simpler
condition, ideally one that yields uniqueness. For particular solutions, such as isolated 1d shocks, 
there are many convenient conditions such as the Lax criterion \cite{lax-claws-ii} or the Liu criterion
\cite{liu-condition-uniqueness-twobytwo,liu-admissibility-memoir}. 
But for large classes of solutions --- such as $L^\infty$ --- the most useful necessary condition is the
\defm{entropy condition} \cite{friedrichs-lax-entropy}: we call a tuple $(\ent,\vec q)$ of smooth functions $\ent,\fen^1,...,\fen^n$
an \defm{entropy-flux pair} if
\begin{alignat*}{5}&
  \fen^i_U(U) = \ent_U(U) A^i(U) \qquad (i=1,...,n)   \myeqlab{eq:eflux}
\end{alignat*}
so that in the case of \emph{smooth} $U$ we obtain
\begin{alignat*}{5}&
    0 \topref{eq:Aform}{=} \ent_U(U)U_t + \ent_U(U) A^i(U)U_i 
    \topref{eq:eflux}{=} \ent(U)_t + \fen^i(U)_i \quad. \myeqlab{eq:smoothe}
\end{alignat*}
For viscous perturbations we have 
\begin{alignat*}{5} 0 
&\topref{eq:vanunifvisc}{=} 
\ent(\Ui)_t + \fen^i(\Ui)_i - \visc \ent_U(\Ui)\Ui_{ii} 
= 
\ent(\Ui)_t + \fen^i(\Ui)_i - \visc \ent(\Ui)_{ii} + \visc \ent_{UU}\Ui_i\Ui_i
\end{alignat*} 
If $\ent$ is convex, then $\ent_{UU}$ is positive semidefinite, so 
\begin{alignat*}{5} 0 &\geq \ent(\Ui)_t + \fen^i(\Ui)_i - \visc \ent(\Ui)_{ii} \end{alignat*} 
Assume $\Ui\conv U$ boundedly a.e.\ as $\visc\conv 0$. In the limit, exploiting that all derivatives are on the outside, we obtain \myeqref{eq:claww} and
\begin{alignat*}{5} 0 &\geq \ent(U)_t + \fen^i(U)_i \myeqlab{ineq:entropy}\end{alignat*} 
in the weak sense. (We have equality for affine $\eta$ (\defm{trivial} entropies).)
A weak solution $U$ of \myeqref{eq:claww} is called \defm{entropy solution} if it satisfies
\myeqref{ineq:entropy}
for \emph{every} entropy-flux pair with \emph{convex} $\ent$. 
For many purposes a single \emph{strictly} convex entropy is sufficient. (There is no guarantee that \emph{any} nontrivial entropies exist.)

For the full Euler equations \myeqref{eq:mass},\myeqref{eq:mom},\myeqref{eq:energy}, 
$\ent:=-\dens s$ with flux $\fen=-\ent s\vv$ is a strictly convex entropy (on the set of $U=(\dens,\dens\vv,E)$ so that $\dens,q>0$);
it is (up to scaling) the negative of the (density of) gas-dynamic entropy in the second law of thermodynamics.

It is well-known (and we show later) that for the isentropic Euler equations \myeqref{eq:mass2},\myeqref{eq:mom2} the \emph{energy}
$\ent:=E$ with entropy fluxes $\fen^i=(E+p)v^i$ takes the role of strictly convex ``entropy''.
While smooth solutions of isentropic Euler are also solutions of the full Euler equations, 
this is no longer true for weak entropy solutions in the presence of shock waves
(for example shocks for isentropic Euler satisfy a strict $E$ inequality
while full Euler shocks conserve $E$).
Likewise, smooth solutions of potential flow are isentropic and therefore full Euler solutions, 
but solutions containing shocks are not (for example multi-d curved Euler shocks generally produce vorticity, 
but potential flow shocks cannot). 
Nevertheless, for solutions with \emph{weak} discontinuities the simplified models yield reasonable approximations;
for example for given upstream data and shock speed, the resulting downstream data is asymptotic 
(see Figure \myref{fig:shockrel} for a comparison of stationary shocks for the three models). 
\begin{figure}
  \hfil\input{shockrel.pstex}\hfil
  \input{shockasy.pstex}\hfil
  \caption{Downstream Mach number $M_+=|\vv_+|/c_+$ as a function of upstream Mach number $M_-=|\vv_-|/c_-$, 
    for $p(\dens)=\dens^\gamma$ with $\gamma=1.4$. Left: the weak shocks are asymptotic but not identical. 
    Right: unlike the approximate models, full Euler has shocks with positive $M_+$ limit.}
  \mylabel{fig:shockrel}
\end{figure}
However, for large variations the discrepancies become significant
and the solutions can at best be considered qualitatively similar.

The second law of thermodynamics motivates one particular entropy inequality directly, without the vanishing viscosity limit;
in fact this motivation applies for perturbations other than uniform or Navier-Stokes viscosity as well. 
Most other ``entropies'' are mathematical devices that may or may not have a particular physical meaning; 
some entropy inequalities --- such as the energy \emph{in}equality for isentropic flow --- are in fact clearly somewhat undesirable 
from a physical point of view and justifiable only as approximations. 

In a \emph{single} space dimension $n=1$, uniqueness --- in various senses --- of entropy solutions
is known for $\BV$ or closely related classes \cite{bressan-crasta-piccoli,bressan-lefloch,liu-yang-2x2,liu-yang}. 
However, in two or more space dimensions 
it appears entropy solutions of the isentropic or full Euler equations need not be unique 
\cite{elling-nuq-journal,de-lellis-szekelyhidi}. 
The author proposes that this is due to the inclusion of vorticity (rotational flow) in an inviscid flow model.
Uniqueness does not and should not hold since inviscid models neglect vorticity on small viscous scales that may propagate
to have a large-scale effect \cite{elling-carbuncle-glimm,elling-hyp2010}. In particular, non-uniqueness is grounded in physics
and therefore searching for more restrictive but reasonable admissibility criteria for Euler flow is futile.
Only models with sufficient viscosity or other smoothing terms to establish a ``bottom'' to the hierarchy of scales
would have a chance for uniqueness. Alternatively, inviscid models could be valid/reliable in cases where the effects of vorticity 
are negligible.
Compressible potential flow is at the top of the hierarchy of zero-vorticity models.

\section{Weak-strong uniqueness and finite speed of propagation}

Beyond general notions of admissibility, which are then the foundation of a large number of results on 1d systems of conservation laws,
such as existence of admissible solutions for some data \cite{glimm,tpliu-deterministic-glimm}, 
strictly convex entropies are also useful for other purposes, 
for example proving \defm{weak-strong uniqueness} using the \defm{relative entropy} idea 
of Dafermos \cite{dafermos-uq} and DiPerna \cite{diperna-uq}. 
To illustrate the idea we prove some simple variations on the theme. 
Here and elsewhere we sacrifice generality for clarity of exposition. 
\begin{theorem}
  \mylabel{th:finitespeed}%
  Assume \myeqref{eq:claww} has a strictly convex entropy-flux pair $(\ent,\fen)$. 
  Let $\Uz$ be a \emph{smooth} solution of \myeqref{eq:claww} with values in a compact convex set $P$.
  Then there is a constant $M<\infty$ with the following property:
  if $U$ is a weak entropy solution of \myeqref{eq:claww} for initial data $U_0$, both with values in $P$, 
  and if 
  \begin{alignat*}{5} U_0 = \Uz(0,\cdot) \quad\text{a.e.\ outside $\cset$} \end{alignat*} 
  where $\cset\subset\R^n$ is closed, then 
  \begin{alignat*}{5} U=\Uz \end{alignat*} 
  a.e.\ outside the cone of influence 
  \begin{alignat*}{5} \coneinf = \selset{(t,x)}{t\geq 0,~d(x,\cset)\leq Mt} \end{alignat*}
\end{theorem}
\begin{remark}
  This theorem unifies the classical weak-strong uniqueness results of Dafermos and DiPerna
  ($\cset=\R^n$) with the classical finite-speed-of-propagation result for symmetric hyperbolic systems
  $$ U_t + A^iU_i = 0 $$
  with constant $A^i$
  (take $f^i(U)=A^iU$, $\ent(U)=\frac12|U|^2$, $\fen^i(U)=\frac12U^TA^iU$). 
\end{remark}
\begin{remark}
  Generally $M$ cannot depend on $\Uz$ alone (physically that would be like ambient air bounding
  the expansion speed of an explosion regardless of its strength, which is obviously unreasonable; 
  a rigorous counterexample can be obtained by solving Riemann problems); it must be permitted to depend on $P$ (hence $U$) as well. 
  
  The theorem cannot be generalized to $\Uz$ that are merely \emph{weak} entropy solutions, 
  since the notion of signal speed makes no sense without uniqueness 
  which is known to be false in higher dimensions \cite{de-lellis-szekelyhidi,elling-nuq-journal}. 
  Uniqueness is also lost if we do not require $U$ to be a weak \emph{entropy} solution. 
\end{remark}

\begin{proof}[Proof of Theorem \myref{th:finitespeed}]
  The relative entropy method seeks an estimate for the quadratic functional
  \begin{alignat*}{5} E(U,\Uz) = \ent(U) - \ent(\Uz) - \ent_U(\Uz)(U-\Uz) \end{alignat*} 
  by taking a distributional $\pt$ (note that $\pt,\pd i$ become classical on expressions not containing $U$):
  \begin{alignat*}{5} E_t 
  &= \ent(U)_t - \ent(\Uz)_t + \subeq{\ent_U(\Uz)\Uz_t}{=\ent(\Uz)_t} - \ent_U(\Uz)U_t - \ent_{UU}(\Uz)\Uz_t(U-\Uz)
  \\&\toprefb{ineq:entropy}{eq:claww}{\leq} -\fen^i(U)_i + \ent_U(\Uz)f^i(U)_i + \ent_{UU}(\Uz)f^i_U(\Uz)\Uz_i(U-\Uz)
  \intertext{(use $\ent_{UU}f^i_U$ symmetric)}
  &= -\fen^i(U)_i + \ent_U(\Uz)f^i(U)_i + \ent_{UU}(\Uz)\Uz_if^i_U(\Uz)(U-\Uz)
  \\&= -\big( \fen^i(U) - \ent_U(\Uz)f^i(U) \big)_i - \ent_U(\Uz)_i \big( f^i(U) - f^i_U(\Uz)(U-\Uz) \big)
  \end{alignat*} 
  To form quadratic functionals we add 
  \begin{alignat*}{5} 0 &= 
    \fen^i(\Uz)_i - \ent_U(\Uz)f^i(\Uz)_i
  \\&=
  \big( \fen^i(\Uz) - \ent_U(\Uz)f^i(\Uz) \big)_i
  + \ent_U(\Uz)_i f^i(\Uz)
  \end{alignat*} 
  so that
  \begin{alignat*}{5} E_t 
  &\leq
  - \Big( \subeq{ \fen^i(U) - \fen^i(\Uz) - \ent_U(\Uz) \big( f^i(U) - f^i(\Uz) \big) }{=\Psi^i} \Big)_i 
  - \subeq{ \ent_U(\Uz)_i \big( f^i(U) - f^i(\Uz) - f^i_U(\Uz)(U-\Uz) \big) }{=R}
  \myeqlab{eq:relent}
  \end{alignat*} 
  By strict convexity of $\ent$ and compactness of $P$ we obtain 
  \begin{alignat*}{5} E \geq \tau|U-\Uz|^2 \end{alignat*} 
  for some constant $\tau=\tau(P)>0$; 
  on the other hand $E,\Psi,R$ are quadratic, so for some $C=C(P)<\infty$
  \begin{alignat*}{5} E,|\Psi|,|R|\leq C|U-\Uz|^2 \leq C\tau^{-1}E \quad. \end{alignat*} 

  \begin{picture}(0,0)%
\includegraphics{conedep.pstex}%
\end{picture}%
\setlength{\unitlength}{3947sp}%
\begingroup\makeatletter\ifx\SetFigFont\undefined%
\gdef\SetFigFont#1#2#3#4#5{%
  \reset@font\fontsize{#1}{#2pt}%
  \fontfamily{#3}\fontseries{#4}\fontshape{#5}%
  \selectfont}%
\fi\endgroup%
\begin{picture}(4156,1554)(1171,-3328)
\put(2551,-2311){\makebox(0,0)[rb]{\smash{{\SetFigFont{6}{7.2}{\rmdefault}{\mddefault}{\updefault}{\color[rgb]{0,0,0}$(x_0,t_0+M^{-1}\delta)$}%
}}}}
\put(1576,-2686){\makebox(0,0)[b]{\smash{{\SetFigFont{6}{7.2}{\rmdefault}{\mddefault}{\updefault}{\color[rgb]{0,0,0}$(\nu^x,\nu^t)$}%
}}}}
\put(3676,-3286){\makebox(0,0)[b]{\smash{{\SetFigFont{6}{7.2}{\rmdefault}{\mddefault}{\updefault}{\color[rgb]{0,0,0}$\delta$}%
}}}}
\put(2551,-2686){\makebox(0,0)[rb]{\smash{{\SetFigFont{6}{7.2}{\rmdefault}{\mddefault}{\updefault}{\color[rgb]{0,0,0}$(x_0,t_0)$}%
}}}}
\put(1201,-1861){\makebox(0,0)[lb]{\smash{{\SetFigFont{6}{7.2}{\rmdefault}{\mddefault}{\updefault}{\color[rgb]{0,0,0}$t$}%
}}}}
\put(4801,-3211){\makebox(0,0)[lb]{\smash{{\SetFigFont{6}{7.2}{\rmdefault}{\mddefault}{\updefault}{\color[rgb]{0,0,0}$x$}%
}}}}
\put(4051,-2686){\makebox(0,0)[lb]{\smash{{\SetFigFont{6}{7.2}{\rmdefault}{\mddefault}{\updefault}{\color[rgb]{0,0,0}$K$}%
}}}}
\put(3376,-2761){\makebox(0,0)[lb]{\smash{{\SetFigFont{6}{7.2}{\rmdefault}{\mddefault}{\updefault}{\color[rgb]{0,0,0}$M$}%
}}}}
\put(3826,-2911){\makebox(0,0)[rb]{\smash{{\SetFigFont{6}{7.2}{\rmdefault}{\mddefault}{\updefault}{\color[rgb]{0,0,0}$1$}%
}}}}
\put(3976,-3286){\makebox(0,0)[lb]{\smash{{\SetFigFont{6}{7.2}{\rmdefault}{\mddefault}{\updefault}{\color[rgb]{0,0,0}$\cset$}%
}}}}
\end{picture}%

  Now let $(x_0,t_0)\notin \coneinf$, so that $\delta=d(x_0,\cset)-Mt_0>0$. 
  Formally we integrate \myeqref{eq:relent} over the cone of dependence of $(x_0,t_0)$, 
  controlling the boundary terms 
  \begin{alignat*}{5} \iint E \nu^t + (\Psi^1,...,\Psi^n)\dotp(\nu^1,...,\nu^n)dS(x)dt \end{alignat*} 
  by taking $M$ large so that $\nu^t\gg|\nu^1|,...,|\nu^n|$,
  while applying a Gronwall argument to the interior. 
%  \begin{alignat*}{5} \pt \iint E ~dx~dt\leq \iint -R~dx~dt \quad \leq \iint C\tau^{-1}E~dx~dt. \end{alignat*} 
  More precisely, define a ``mollified cone of dependence'':
  \begin{alignat*}{5} 
    \chi(x,t) &= \begin{cases}
      1, & |x-x_0| \leq M(t_0-t), \\
      1-(|x-x_0|-M(t_0-t))\delta^{-1} , & M(t_0-t) \leq |x-x_0| \leq M(t_0-t)+\delta \\
      0, & M(t_0-t)+\delta \leq |x-x_0|
    \end{cases}
    \end{alignat*} 
    and
    $\vartheta(t) = e^{-ct}$. 
    The weak formulation extends from smooth to Lipschitz test functions like $\chi\vartheta\geq 0$, so
    \begin{alignat*}{5}
      0 &\leq
      \int_\R \vartheta(0)\chi(x,0) \subeq{ E(U_0(x),\Uz(0,x)) }{=0} dx 
      \\&+ \int_0^\infty \int_{M(t_0-t)\leq|x-x_0|\leq M(t_0-t)+\delta}
      \subeq{\vartheta}{>0}\big(
      \subeq{\chi_t}{\leq -M\delta^{-1}} 
      \subeq{ E(U,\Uz) }{\geq 0} 
      +
      \subeq{\nabla\chi}{|\cdot|\leq\delta^{-1}} \dotp \subeq{ \Psi(U,\Uz) }{|\cdot|\leq C\tau^{-1} E(U,\Uz)}
      \big)
      dx~dt
      \\&+ \int_0^\infty \int_{|x-x_0|\leq M(t_0-t)+\delta}
      \chi \big( \subeq{\vartheta_t}{=-c\vartheta}
      \subeq{ E(U,\Uz) }{\geq 0} 
      \subeq{ -R(U,\Uz) }{\leq C\tau^{-1}E(U,\Uz)}
      \vartheta 
      \big)
      dx~dt
      \myeqlab{eq:weakform}
    \end{alignat*}
    The first integral is zero because $\chi(0,x)\neq 0$ means $|x-x_0|< Mt_0+\delta=d(x_0,\cset)$, so $x\notin \cset$, 
    hence $\Uz(0,x)=U_0(x)$. 
    Upon picking $M>C\tau^{-1}$ and $c\geq C\tau^{-1}$ we obtain that the second and third integral are $\leq 0$, 
    and in particular considering the $\chi_t$ term we have $E(U,\Uz)=0$ a.e.\ in $\supp\chi_t$ (light shaded area in the picture). 
    These regions, upon varying $(x_0,t_0)$, cover the complement of $\coneinf$, where we obtain that $U=\Uz$ a.e.
\end{proof}

\section{Convergence of higher-order perturbations to classical solutions}

\begin{theorem}
  \mylabel{th:vanviscconv}%
  Consider as domain a torus $\T^n$. 
  Assume \myeqref{eq:claww} has a strictly convex entropy-flux pair $(\eta,\fen)$. 
  Let $\Uz$ be a \emph{smooth} solution of \myeqref{eq:claww} with values in a convex compact $P\subset\R^m$.
  For some sequence of $\visc\conv 0$ and for \emph{same} initial data 
  let there be weak (measurable) solutions 
  $\Ui$ with values in $P$ of 
  \begin{alignat*}{5} 0 =& \Ui_t + f^i(\Ui)_i - \visc L\Ui  
  \end{alignat*} 
  where $L$ is a continuous operator on $\Ddist$, and assume
  \begin{alignat*}{5} 0 \geq& \ent(\Ui)_t + \fen^i(\Ui)_i \quad. \myeqlab{eq:entropy-compatible} \end{alignat*} 
  Then 
  \begin{alignat*}{5} \Ui\conv\Uz \quad\text{a.e.\ as $\visc\conv 0$} \end{alignat*} 
\end{theorem}
\begin{proof}
  We repeat the steps in the derivation of \myeqref{eq:relent}, with additional terms:
  using \myeqref{eq:entropy-compatible} we get
  \begin{alignat*}{5} E(\Ui,\Uz)_t + \Psi^i(\Ui,\Uz)_i + \subeq{R(\Ui,\Uz)}{\geq-C\tau^{-1}E} &\leq 
  - \visc~ \ent_U(\Uz) L\Ui
  \conv 0 \quad\text{in $\mathcal{D}'$ as $\visc\conv 0$}
  \end{alignat*} 
  because $(\Ui)$ is bounded in $\Linf\embed\Ddist$, 
  so $(\ent_U(\Uz)L\Ui)$ is also bounded because $L$ and (by smoothness of $\Uz$)
  multiplication by $\ent_U(\Uz)$ are continuous on $\Ddist$.
  After integrating over $\T^n$,
  a standard Gronwall argument using $E(\Ui,\Uz)=0$ at $t=0$ shows $\int E(\Ui,\Uz)dx\conv 0$, 
  hence $\Ui\conv\Uz$ a.e.
\end{proof}
\begin{remark}
  The most obvious application is the uniform viscosity perturbation $L=\Lap$, where \myeqref{eq:entropy-compatible} is implied by
  \begin{alignat*}{5} -\visc \ent_U(U)U_{ii} = \visc \big( \ent(U)_{ii} - \ent_{UU}(U)U_iU_i \big) 
  \overset{\ent_{UU}\geq0}{\leq} \visc \ent(U)_{ii} \conv 0 \quad\text{in $\mathcal{D}'$ as $\visc\dnconv 0$.} \end{alignat*} 
  Of course for this uniform viscosity perturbation there are other ways of obtaining $\Ui\conv\Uz$, 
  for example inverting $\pt-\visc\Lap$ to a heat kernel $G$ and estimating a right-hand side $f^i(\Ui)_i*G(\eps t)$. 
  But if boundedness is already known, the relative entropy approach is shorter and applies to any $L$.
  The entropy inequality \myeqref{eq:entropy-compatible} is satisfied for the physical entropy 
  if the perturbation $\visc L$ is sufficiently realistic. 
\end{remark}
\begin{remark}
  When initial data only matches outside a set $\cset$, 
  convergence outside a cone of dependence $\coneinf$ can be shown by combining the proofs of Theorem \myref{th:finitespeed}
  and Theorem \myref{th:vanviscconv}.
  This requires some care:
  $\chi$ must be modified to be smooth,
  and a single constant $M$ suffices only if $\cset$ is compact or further assumptions on $L$ are made.
\end{remark}

\section{Potential flow as a first-order system}

Compressible potential flow \myeqref{eq:potf} is a second-order PDE. 
Although many of the results on theory and numerics of 1st order systems may be adaptable, 
which of them are is not immediately obvious,  
and the sheer amount makes the job rather tedious. 
We provide a shortcut by writing potential flow as an equivalent first-order system. 

Observing that for \emph{steady} potential flow the density $\dens$ determines the magnitude $|\vv|$ of the velocity
via the steady Bernoulli relation
\begin{alignat*}{5} 0 &\topref{eq:Bern}{=} \frac12|\vv|^2 + \piv(\dens) \quad, \end{alignat*} 
some past work \cite{osher-hafez-whitlow,morawetz-1995,chen-dafermos-slemrod-wang} 
has studied a $2\times 2$ system for $\dens$ and the angle $\theta$ of $\vv$. 
\cite{osher-hafez-whitlow} use these variables as the basis for an entropy and propose an admissibility criterion. 
$\dens,\theta$ are unattractive for our goal: generalization to $n\geq 3$ seems awkward (although possible),
$\vv=0$ becomes a singularity, but most importantly the idea does not generalize to unsteady flow. 

Another line of work \cite{chipman-jameson-1979} 
employs the gradient \myeqref{eq:vB} of the unsteady Bernoulli relation \myeqref{eq:bernoulli}:
\begin{alignat*}{5}
    0 &= \dens_t + \nabla\cdot(\dens\vv), \myeqlab{eq:rho3} \\
    0 &= \vv_t + \nabla\big(\frac12|\vv|^2+\pi(\dens)\big) \myeqlab{eq:rhovk}
\end{alignat*}
with conserved density $U=(\dens,\vv)$. 
The distributional curl of \myeqref{eq:rhovk} immediately implies 
\begin{alignat*}{5} 0 &= \pt \curl\vv \end{alignat*} 
If the initial $\vv$ is also curl-free, then $\curl\vv=0$ for all $t$, so $\vv=\nabla\phi$ in the distributional sense. 
In that case weak solutions of the system above correspond to weak solutions of compressible potential flow \myeqref{eq:potf}.
The opposite direction is true as well.

While the remaining discussion applies to $n\geq 3$ with obvious modifications, 
we focus on $n=2$ (with $(v,w)=(v^1,v^2)$, $(x,y)=(x^1,x^2)$) for readability.  
In contrast to the isentropic Euler system
\begin{alignat*}{5} 
0 &= \begin{bmatrix}
  \dens \\ v \\ w
\end{bmatrix}_t
+
\begin{bmatrix}
  v & \dens & 0 \\
  \piv_\dens & v & \mathbf 0 \\
  0 & 0 & \mathbf v
\end{bmatrix}
\begin{bmatrix}
  \dens \\ v \\ w
\end{bmatrix}_x
+
\begin{bmatrix}
  w & 0 & \dens \\
  0 & \mathbf w & 0 \\
  \piv_\dens & \mathbf 0 & w \\
\end{bmatrix}
\begin{bmatrix}
  \dens \\ v \\ w
\end{bmatrix}_y
\end{alignat*} 
the 1st order system form of compressible potential flow reads
\begin{alignat*}{5} 
0 &= \begin{bmatrix}
  \dens \\ v \\ w
\end{bmatrix}_t
+
\subeq{\begin{bmatrix}
  v & \dens & 0 \\
  \piv_\dens & v & \mathbf w \\
  0 & 0 & \mathbf 0
\end{bmatrix}}{=\begin{bmatrix}
  A^x_\dens & A^x_v & A^x_w
  \end{bmatrix}}
\begin{bmatrix}
  \dens \\ v \\ w
\end{bmatrix}_x
+
\subeq{\begin{bmatrix}
  w & 0 & \dens \\
  0 & \mathbf 0 & 0 \\
  \piv_\dens & \mathbf v & w \\
\end{bmatrix}}{=\begin{bmatrix}
  A^y_\dens & A^y_v & A^y_w
  \end{bmatrix}}
\begin{bmatrix}
  \dens \\ v \\ w
\end{bmatrix}_y
\myeqlab{eq:potfA}
\end{alignat*} 
This is not particularly attractive either. While the latter system is still rotation-invariant, 
it is no longer Galilean-invariant: seeking perturbations $\tilde U=\tilde U(t,x)$ by linearization
yields characteristics $x/t=\lambda^\alpha$ where the eigenvalues $\lambda^\alpha$ are
\begin{alignat*}{5} v-c , 0 , v+c \end{alignat*} 
as opposed to $v-c,v,v+c$ for isentropic. 
As a result the system is strictly hyperbolic only when $|\vv|\neq c$. 
It would be worthwhile to find a different 1st order system form of potential flow that does not suffer from these defects.

\section{Entropies for potential flow and isentropic Euler}

We determine all entropies for \myeqref{eq:potfA}, recovering the known results for 1d and for isentropic Euler along the way.

\subsection{1d case}

The entropy-flux condition $\fen^i_U=\eta_UA^i$ in components (with $A^i$ columns $A^i_\dens,A^i_v,A^i_w$):
\begin{alignat*}{5} 
\fen^x_\dens &= \ent_U A^x_\dens 
= v\ent_\dens + \piv_\dens\ent_v 
\\
\fen^x_v &= \ent_U A^x_v 
= \dens \ent_\dens + v \ent_v 
\end{alignat*} 
Compatibility relations:
\begin{alignat*}{5} 0 
&= (\fen^x_\dens)_v-(\fen^x_v)_\dens 
= (v\ent_\dens + \piv_\dens\ent_v)_v - (\dens \ent_\dens + v \ent_v)_\dens
= \ent_\dens + v\ent_{\dens v} + \piv_\dens\ent_{vv} - \ent_\dens - \dens \ent_{\dens\dens} - v \ent_{\dens v}
\\&= \piv_\dens\ent_{vv} - \dens \ent_{\dens\dens} 
\myeqlab{eq:rel1d}
\end{alignat*} 
In 1d this is the only condition on $\eta$ for existence of $\fen^x$, and it is the same for potential flow and for isentropic flow. 
We easily obtain a large family of entropies, for example by imposing $\ent$ and $\ent_\dens$ on $v=0$
and solving \myeqref{eq:rel1d} locally. A particular choice is 
\begin{alignat*}{5}&
    \ent(\dens,v) = \frac12 v^2 + g(\dens)  \myeqlab{eq:g}
\end{alignat*}
where $g$ is chosen to solve $g_{\dens\dens}(\dens)=\dens^{-1}\piv_\dens(\dens)=c^2\dens^{-2}>0$; clearly $g$ and hence $\ent$ 
are strictly convex.

\subsection{2d case}

For isentropic Euler
\begin{alignat*}{5} 
\fen^x_w &= \ent_UA^x_w = v \ent_w
&\csep
\fen^y_v &= \ent_UA^y_v = w \ent_v
\end{alignat*} 
The first yields the compatibility relation
\begin{alignat*}{5} 0 = (\fen^x_w)_\dens - (\fen^x_\dens)_w = (v\ent_w)_\dens - (v\ent_\dens+\piv_\dens\ent_v)_w = -\subeq{\piv_\dens}{>0} \ent_{v w} 
\quad\impl\quad \ent_{v w}=0 \end{alignat*} 
and it implies the analogous relation for $y$. For potential flow on the other hand we have
\begin{alignat*}{5} 
\fen^x_w &= \ent_UA^x_w = w \ent_v
&\csep
\fen^y_v &= \ent_UA^y_v = v \ent_w
\end{alignat*} 
These yield
\begin{alignat*}{5} 0 
&= (\fen^x_w)_\dens - (\fen^x_\dens)_w 
= (w\ent_v)_\dens - (v\ent_\dens+\piv_\dens\ent_v)_w 
= w\ent_{v\dens} - v\ent_{w\dens} - \piv_\dens\ent_{vw} \quad;
\end{alignat*} 
the analogous relation for $y$ is 
\begin{alignat*}{5} 
0 & = v\ent_{w\dens} - w\ent_{v\dens} - \piv_\dens\ent_{vw} 
\end{alignat*} 
and taking the sum and using $\piv_\dens>0$ we obtain 
\begin{alignat*}{5} 0 &= \ent_{vw} \quad, \end{alignat*} 
as for isentropic Euler. 

Integrating this over $v$ and then $w$ yields 
\begin{alignat*}{5} \ent(\dens,v,w) &= g^x(\dens,v) + g^y(\dens,w) \myeqlab{eq:ent-sum} \end{alignat*} 
Moreover, in 2d the relation \myeqref{eq:rel1d} has the $x\rightarrow y$ analogue 
\begin{alignat*}{5} 
0 &= 
\piv_\dens\ent_{ww}-\dens\ent_{\dens\dens}
\end{alignat*} 
and both combined yield
\begin{alignat*}{5} \ent_{vv} &= \ent_{ww} \end{alignat*} 
Applied to \myeqref{eq:ent-sum} this yields
\begin{alignat*}{5} g^x_{vv}(\dens,v) &=  g^y_{ww}(\dens,w) \quad. \end{alignat*} 
That shows both sides are constant in $v,w$:
\begin{alignat*}{5} g^x_{vv}(\dens,v) = g^E(\dens) = g^y_{ww}(\dens,w) \end{alignat*} 
so integrating twice we obtain
\begin{alignat*}{5} \ent(\dens,v,w) &= g^0(\dens) + g^v(\dens)v + g^w(\dens)w + g^E(\dens)\frac{v^2+w^2}{2} \myeqlab{eq:entb} \end{alignat*} 
for some functions $g^0,g^v,g^w$.

The last two compatibility relations: for isentropic flow,
\begin{alignat*}{5} 0 &= (\fen^x_v)_w-(\fen^x_w)_v = (\dens\ent_\dens+v\ent_v)_w-(v\ent_w)_v = \dens\ent_{\dens w} - \ent_w 
= \dens g^w_\dens(\dens)+ \dens g^E_\dens(\dens) w - g^w(\dens) - g^E(\dens) w  \quad, \end{alignat*} 
so by varying $w$ we see that 
\begin{alignat*}{5} \dens g^w_\dens(\dens) - g^w(\dens) &= 0 \quad\impl\quad g^w(\dens) = C^w\dens \quad, \\
\dens g^E_\dens(\dens) - g^E(\dens) &=  0 \quad\impl\quad g^E(\dens) = C^E\dens \quad. \end{alignat*} 
Using the analogous $y$ relation we obtain $g^v(\dens) = C^v\dens$. 

For potential flow, we instead have 
\begin{alignat*}{5} 0 &= (\fen^x_v)_w-(\fen^x_w)_v = (\dens\ent_\dens+v\ent_v)_w-(w\ent_v)_v 
\overset{\ent_{vw}=0}{=} \dens\ent_{\dens w} - w \ent_{vv} 
= \dens g^w_\dens(\dens) + \dens g^E_\dens(\dens) w - w g^E(\dens) \end{alignat*} 
Varying $w$ we obtain again that $g^E(\dens)=C^E\dens$, but $g^w(\dens)=C^w=\const$ now; 
the analogous $y$ relation yields $g^v(\dens)=C^v=\const$.

Either way $g^v_{\dens\dens}=g^w_{\dens\dens}=g^E_{\dens\dens}=0$, so the last remaining equation \myeqref{eq:rel1d} becomes
\begin{alignat*}{5} 0 &\toprefb{eq:rel1d}{eq:entb}{=} 
\piv_\dens C^E \dens
- \dens g^0_{\dens\dens}(\dens) \quad\impl\quad g^0(\dens) = C^0 + C^\dens\dens + C^E \int^\dens\piv(\dens)d\dens \end{alignat*} 

Altogether, for isentropic Euler we have
\begin{alignat*}{5} \ent &= C^0 + C^\dens\dens + C^v\dens v + C^w\dens w + C^E ( \subeq{ \int^\dens\piv(\dens)d\dens + \dens\frac{v^2+w^2}{2} }{=E} ) \end{alignat*} 
whereas for potential flow we obtain
\begin{alignat*}{5} \ent &= C^0 + C^\dens\dens + C^v v + C^w w + C^E E \end{alignat*} 
Either way, any entropy is a linear combination of $1$ and the conserved quantities ($\rho,\rho v,\rho w$ for isentropic Euler,
$\rho,v,w$ for potential flow), and the energy $E$, which is the only nontrivial part.

\subsection{Convexity}

For isentropic Euler the Hessian of 
\begin{alignat*}{5} E(\dens,\dens v,\dens w) = \int^\dens\piv(\dens)d\dens + \frac{(\dens v)^2+(\dens w)^2}{2\dens} \end{alignat*} 
is
\begin{alignat*}{5} 
\begin{bmatrix}
  \piv_\dens + \dens^{-1} (v^2+w^2) & - \dens^{-1} v & - \dens^{-1} w \\
  - \dens^{-1} v & \dens^{-1} & 0 \\
  - \dens^{-1} w & 0 & \dens^{-1}
\end{bmatrix} \end{alignat*} 
The matrix has determinant $ \dens^{-2} \piv_\dens > 0$, and so do the lower right
$1\times 1$ and $2\times 2$ blocks, so the Hessian is positive definite 
and $E$ is strictly convex. 

For potential flow the Hessian of
\begin{alignat*}{5} E(\dens,v,w) = \int^\dens\piv(\dens)d\dens + \dens \frac{v^2+w^2}{2} \end{alignat*} 
is
\begin{alignat*}{5} 
\begin{bmatrix}
  \piv_\dens & v & w \\
  v & \dens & 0 \\
  w & 0 & \dens
\end{bmatrix} \end{alignat*} 
The lower right $1\times 1$ and $2\times 2$ block have positive determinant, but the entire matrix 
has determinant $\dens(c^2-v^2-w^2)$.
This is positive, and $E$ is locally strictly convex, if and only if $\vv$ is subsonic: $|\vv|<c$. 

In the case of $n\geq 3$ space dimensions, the compatibility relations for higher-dimensional isentropic Euler 
resp.\ potential flow imply the 2d forms when setting $v^3=...=v^n=0$. 
Thus the set of entropies cannot be larger, and it is straightforward to check that $E$ remains an entropy
with same convexity properties. 

The restriction $|\vv|<c$ can be relaxed: the weak form of 2nd order compressible potential flow itself
is Galilean-invariant, so we may change to any frame moving at constant speed before passing to the 1d system \myeqref{eq:rho3},\myeqref{eq:rhovk}. 
Hence, for flows that have local velocity \emph{variation} less than twice the speed of sound $c$ (for non-isothermal flow the velocity dependence of $c$
must also be considered), 
we obtain enough entropies to define admissibility \emph{locally}. 

Nevertheless, the energy is not a suitable entropy in cases of local velocity variation larger than the speed of sound. 
As Figure \myref{fig:shockrel} right shows, such shocks are certainly mathematically possible (if physically inaccurate). 
They would certainly occur if, for example, at initial time the fluid adjacent to a solid wall has normal velocity $v^n>c$;
at $t>0$ a shock would emanate from the wall, separating a $v^n>c$ region from a $v^n=0$ region. 
Such flows have been studied in \cite{elling-liu-pmeyer} in the context of supersonic flow onto a solid wedge and in
\cite{elling-rrefl-lax} in the context of regular reflection. 

The 1d entropy derived in \myeqref{eq:g} does not suffer from this restriction, 
but as we have shown it is not an entropy in dimensions $n\geq 2$. 

\begin{theorem}
  For compressible potential flow in the form \myeqref{eq:rho3},\myeqref{eq:rhovk}, 
  energy $E$ is the only nontrivial entropy. It is strictly convex in $(\rho,\vv)$ for $|\vv|<c$. 
\end{theorem}

\section{Conclusion}

The energy inequality, while not quite satisfactory, is currently the only simple admissibility criterion 
for unsteady compressible potential flow applicable to general function spaces. 

Vorticity appears to be the root of non-uniqueness in the Euler case,
but is absent in potential flow. We conjecture the following:\\
\centerline{\fbox{\parbox{\linewidth-2\fboxsep-2cm}{For given irrotational initial data $\dens_0,\vv_0\in L^\infty$ ($\dens_0>0$), 
      there is at most one admissible weak solution of compressible potential flow.}}}
Some qualifications are likely needed, for example restrictions on the pressure law.

\bibliographystyle{amsalpha}
\bibliography{../../../research/pmeyer/elling}

\end{document}